\documentclass{amsart}
\usepackage[utf8]{inputenc}
\usepackage[T1]{fontenc}
\usepackage{todonotes}
\usepackage{amsmath,amsfonts,amssymb,amsthm}
\usepackage{hyperref}
\usepackage{cleveref}
\usepackage[all]{xy}
\usepackage{longtable}

\usepackage{multirow}
\usepackage{enumerate}

\newtheorem{thm}{Theorem}[section]

\newtheorem{thmINTRO}{Theorem}

\newtheorem{prop}[thm]{Proposition}
\newtheorem{lemma}[thm]{Lemma}
\newtheorem{rmk}[thm]{Remark}
\newtheorem{example}[thm]{Example}

\newcommand{\rk}{\mathrm{rk}\,}
\newcommand{\pic}{\mathrm{NS}(X)}
\newcommand{\Tx}{\mathrm{T}_X}
\newcommand{\LL}{\mathbb{L}_{20}}
\newcommand{\Z}{\mathbb{Z}}
\newcommand{\quotient}[2]{{\raisebox{.2em}{$#1$}\left/\raisebox{-.2em}{$#2$}\right.}}
\renewcommand{\tilde}{\overset{\sim}}

\title{K3 surfaces with action of the group $M_{20}$}
\author{Paola Comparin, Romain Demelle}
\begin{document}

\maketitle

\begin{abstract}
It was shown by Mukai that the maximum order of a finite group acting faithfully and symplectically on a K3 surface is 960 and if such a group has order 960, then it is isomorphic to the Mathieu group $M_{20}$. In this paper, we are interested in projective K3 surfaces admitting a faithful symplectic action of the group $M_{20}$. We show that there are infinitely many K3 surfaces with this action and we describe them and their projective models, giving some explicit examples.
\end{abstract}

\section*{Introduction}
A {\it K3 surface} $X$ is a compact complex smooth surface which is simply connected and admits a nowhere vanishing holomorphic 2-form $\omega_X$, unique up to scalar multiplication. 
The study of finite groups acting on K3 surfaces and their action in cohomology goes back to the first works by Nikulin and Mukai. In particular, given a finite group $G$ acting on a K3 surface $X$, one can study the action of $G$ according to the action on the 2-form $\omega_X$: an automorphism $\sigma\in \mathrm{Aut}(X)$ is said to be {\it symplectic} if it acts as the identity on $\omega_X$ and {\it non-symplectic} otherwise. In this last case $\sigma^*$ acts as multiplication by some root of unity $\zeta_m\in\mathbb C^{\times}$.
Given a finite group of automorphisms $G$, let $\alpha:G\rightarrow\mathbb C^{\times}$ corresponding to the action on $\omega_X$. The following short exact sequence holds for some positive integer $m$:

$$\xymatrix{1\ar[r]& G_0\ar[r] & G\ar[r]^-{\alpha}&\mathbb Z/m\mathbb Z\ar[r] & 1}$$

where $G_0=\ker \alpha$ corresponds to the automorphisms acting symplectically on $X$.
A first result in the classification of possible groups $G$ acting on a K3 surface is given by Nikulin in \cite{Nikulin}, where the author classified abelian finite groups acting symplectically and showed  that there is only a finite number of them. 

Without the assumption of being abelian, Mukai in \cite{Mukai} studied finite groups acting faithfully and symplectically on K3 surfaces and showed that every such group is a subgroup of the Mathieu group $M_{23}$. Xiao in \cite{xiao} gave an explicit list of all possible finite groups acting symplectically, finding 81, and classified combinatorial types of the action of these groups, i.e. computed the number of fixed points in each case.

Mukai had also shown that the order of such a group is at most 960, being equal to 960 if and only if $G$ is the Mathieu group $M_{20}$, that one can describe also as $A_5\ltimes(\mathbb{Z}/2\mathbb{Z})^4$.

In this paper we are concerned with projective K3 surfaces admitting a faithful and symplectic action of $M_{20}$. We first remark that a classification of possibilities for maximal groups $\Gamma$ acting on a K3 surface $X$ and such that $\Gamma$ contains $M_{20}$ strictly is given in \cite{BS} and \cite{BH}. 
There are three such groups and their order is $3840=4\cdot |M_{20}|$ in one case and $1920=2\cdot |M_{20}|$ in the two other cases. The K3 surfaces where these groups act first appeared in \cite{BS,BH,Kondo,Mukai}.

Let $X$ be a K3 surface that admits a faithful and symplectic action of the Mathieu group $M_{20}$. The surface $X$ has necessarily Picard number $\rho(X):=\rk\pic$ equal to 20 by Xiao's classification \cite{xiao}. 
A first question arising is how many K3 surfaces admit a symplectic and faithful action of $M_{20}$. By \cite[Proposition 2.1]{BS} it is already known that there are at most countably many such surfaces. In Section \ref{infinity} we prove the following further result.

\begin{thmINTRO}\label{mainTH1}
There exists an infinite number of K3 surfaces admitting a faithful and symplectic action of $M_{20}$.
\end{thmINTRO}

Let $\mathbb L_{20}$ be the lattice defined by the following matrix:
$$\LL=\left(\begin{array}{ccc} 4 & 0 & -2 \\ 0 & 4 & -2 \\ -2 & -2 & 12 \end{array}\right).$$

We show in Section \ref{sec-background} some properties of this lattice and its link with the invariant lattice $\mathrm{H}^2(X,\mathbb{Z})^{M_{20}}$ and the transcendental lattice $T_X=\pic^{\perp_{\mathrm{H}^2(X,\mathbb Z)}}$.
We consider $L$ a primitive element of $\LL$ invariant by the action of $M_{20}$. As the square of all element in $\LL$ is a multiple of 4, we automatically get $L^2=4n$, for some $n\in\mathbb{Z}_{\geq1}$. The second main result of the paper is the following:

\begin{thmINTRO}\label{mainTH2}
There exists an embedding $\langle 4n\rangle\hookrightarrow\LL$ if and only if $n$ is not of the form $4^i(16j+6)$, with $i,j$ some non-negative integers. If it is the case, we can construct a projective model of the surface $X$ in $\mathbb{P}^{2n+1}$ and we have the following results:
\begin{enumerate}[1.]
    \item $L$ is an ample class;
    \item the linear system $\vert L\vert$ is not hyperelliptic and is base point free;
    \item for $n>1$, the projective model is only defined by quadrics.
\end{enumerate}
Morevover, in this last case, the number of quadrics $Q_{4n}$ defining $X$ is
$$Q_{4n}=2n^2-3n+1.$$
\end{thmINTRO}

The paper is organized as follows: in Section \ref{sec-background}, we recall some notations and results we use throughout the paper. In order to illustrate Theorem \ref{mainTH2}, we start with some particular cases. The cases $L^2=4n$ with $n=1,2,10$ have been already studied in \cite{BS,Kondo, BH}, thus in Section \ref{L^2=12} we begin by studying the first new case which is $n=3$, that is $L^2=12$. We prove the existence of the embedding and we construct explicitly the projective model. In Section \ref{other_cases} we cover the cases $L^2=4n$ with $4\leq n\leq 10$: we show the cases we cannot construct the embedding for and see possible behaviours such as the existence of different K3 surfaces for the same $n$ or multiple embeddings on a same K3
surface. 
All the results we obtain are summarized in Appendix \ref{tab:T_X}. 
Section \ref{projective_model} is dedicated to the proof of Theorem \ref{mainTH2}. Thanks to the results in \cite{SD} we show that the linear system $\vert L\vert$ is base point free, is not hyperelliptic and the projective model is defined only by quadrics.
In Section \ref{prim} we are interested in some non-primitive embeddings, where the non-primitiveness allows to find explicit equations for the surfaces. Finally, in Section \ref{infinity} we prove Theorem \ref{mainTH1}.
 
\section*{Acknowledgments}
We would like to thank C\'edric Bonnaf\'e, Xavier Roulleau and Alessandra Sarti for very useful insights and discussions. We thanks the anonymous referee for useful comments which allows to simplify the proofs in Section 4.
The authors have been partially supported by Fondecyt Iniciaci\'on en la investigaci\'on N.11190428, Programa de cooperaci\'on cient\'ifica ECOS-ANID C19E06 and Math AmSud-ANID 21 Math 02. The first author has been partially supported by Fondecyt Regular N.1200608.

\section{Notations and preliminaries}\label{sec-background}

Let $X$ be a {\it K3 surface} and denote by $\omega_X$ the nowhere vanishing holomorphic 2-form. We have $\mathrm{H}^{2,0}(X):=\mathrm{H}^0(X,\Omega_X^2)=\mathbb{C}\cdot\omega_X$ and the cohomology group $\mathrm{H}^2(X,\mathbb{Z})$ is isometric to the {\it K3 lattice} $\Lambda_{K3}$:
$$\Lambda_{K3}:=U^{\oplus 3}\oplus E_8(-1)^{\oplus 2}$$
where $U$ is the hyperbolic plane and $E_8(-1)$ is the negative definite lattice associated to the root system with the same name. 

Let $G_0$ be a group of {\it symplectic automorphisms} of $X$, that is the automorphisms acting as the identity on the 2-form $\omega_X$. We are interested in the case where $G_0$ is the Mathieu group $M_{20}$. In this paper, we want to describe all projective K3 surfaces admitting a symplectic and faithful action of $M_{20}$.

The lattice $\LL$ we defined in the Introduction has rank $3$ and signature $(3,0)$. Denoting by $\mathrm{H}^2(X,\mathbb{Z})^{M_{20}}$ the invariant lattice by the action of $M_{20}$ on the K3 lattice, we recall the following result:  

\begin{prop}{(\cite[proof of Proposition 2.1]{Kondo})}
Let $X$ be a K3 surface with a faithful symplectic action by $M_{20}$. Then the invariant lattice $\mathrm{H}^2(X,\mathbb{Z})^{M_{20}}$ is isometric to $\LL$.
\end{prop}

As $\LL$ has signature $(3,0)$, its isometry group $\mathcal{O}\left(\LL\right)$ is finite, see \cite[proof of Prop. 2.1]{Kondo}. We describe it completely thanks to the following result coming from \cite[Remark 2.4]{BS}:

\begin{prop}\label{isometries}
Denote by $\rho_1$ and $\rho_2$ the isometries of $\LL$ defined by the following matrices:

$$\rho_1=\left(\begin{array}{ccc} 0 & 1 & 0 \\ 1 & 0 & 0 \\ 0 & 0 & 1\end{array}\right) \quad\mathrm{and}\quad \rho_2=\left(\begin{array}{ccc} 1 & 0 & -1 \\ 0 & -1 & 0 \\ 0 & 0 & -1\end{array}\right)\, .$$

Then the isometry group $\mathcal{O}\left(\LL\right)$ of $\LL$ is spanned by  $-\mathrm{id}_{\LL}$, $\rho_1$ and $\rho_2$. Moreover, $\mathcal{O}\left(\LL\right)$ has order $16$.
\end{prop}

By \cite[Nr. 81, Table 2]{xiao}, as the K3 surface $X$ admits a faithful symplectic action of $M_{20}$, the minimal resolution of the quotient of $X$ by $M_{20}$ is a K3 surface with Picard number 20. By a result of Inose \cite[Corollary 1.2]{inose}, this also means that $X$ has Picard number 20. Denoting by $\Tx$ the {\it transcendental lattice} of $X$, i.e. the orthogonal complement of $\pic$ in $\mathrm{H}^2(X,\mathbb{Z})$, we have $\rk\Tx =2$. This lattice is even, with signature $(2,0)$. In order to have an explicit description of $T_X$, we recall some facts about K3 surfaces with Picard number $20$, which are called {\it singular K3 surfaces} (cf. \cite[Section 4]{SI}). Denote by $\mathcal{Q}$ the set of $2\times 2$ positive-definite even integral matrices. For $Q\in\mathcal{Q}$ we have:
$$Q=\left(\begin{array}{cc} 2a & b \\ b & 2c \end{array}\right),\, a,b,c\in\mathbb{Z}$$

with $a,c>0$ and $d:=4ac-b^2>0$. We define an equivalence relation $\sim$ on $\mathcal{Q}$ by:
$$\forall\  Q_1,Q_2\in\mathcal{Q},\, Q_1\sim Q_2\Longleftrightarrow\exists\gamma\in\mathrm{SL}_2(\mathbb{Z})\ ,\, Q_1=\,^t\gamma Q_2\gamma.$$
Let $\left[Q\right]$ be the equivalence class of a matrix $Q\in\mathcal{Q}$ and $\quotient{\mathcal{Q}}{\mathrm{SL}_2(\mathbb{Z})}$ the set of equivalence classes. 

\begin{thm}{(\cite[Theorem 4]{SI})}\label{T_X}
The map $X\mapsto\left[T_X\right]$ establishes a bijective correspondence from the set of singular K3 surfaces onto $\quotient{\mathcal{Q}}{\mathrm{SL}_2(\mathbb{Z})}$.
\end{thm}

In particular, K3 surfaces with Picard number $20$ are classified in terms of their transcendental lattice. Moreover, we can choose a representative $Q$ in {\it reduced form}, that is $-a\leq b\leq a\leq c$ and $b^2\leq ac\leq\frac{d}{3}$ and we have the following:

\begin{thm}{(\cite[Theorem 2.4]{buell})}
No distinct reduced quadratic forms are equivalent, except for the following:
$$\left(\begin{array}{cc}
2a & a \\ a & 2c
\end{array}\right)\sim\left(\begin{array}{cc}
2a & -a \\ -a & 2c \end{array}\right)
\quad\mathrm{and}\quad
\left(\begin{array}{cc}
2a & b \\ b & 2a
\end{array}\right)\sim\left(\begin{array}{cc}
2a & -b \\ -b & 2a \end{array}\right).
$$
\end{thm}

Since the action is symplectic one can show that $\Tx\subset\mathrm{H}^2(X,\mathbb{Z})^{M_{20}}\simeq\LL$.
In fact, consider $x\in T_X\subset\mathrm{H}^2(X,\mathbb{Z})$ and $g\in M_{20}$. We have
$$\langle\omega_X,x\rangle=\langle g^*\omega_X,g^*x\rangle=\langle \omega_X,g^*x\rangle$$
so that $\langle\omega_X,x-g^*x\rangle=0$. Then $x-g^*x$ is in $(\mathbb{C}\omega_X)^{\perp}\cap\mathrm{H}^2(X,\mathbb{Z})$ which is $\pic$.
Moreover, $x\in T_X$ and $g^*x\in T_X$ because $g^*$ is a Hodge isometry, so it preserves the lattice $T_X$. Hence $x-g^*x\in\pic\cap T_X=\lbrace0\rbrace$ by \cite[Sec. 3.2]{Nikulin} which implies that $g^*x=x$. So $T_X$ is included in the invariant lattice $\mathrm{H}^2(X,\mathbb{Z})^{M_{20}}$ which is isomorphic to $\LL$.\\

If $(u,v)$ is a $\mathbb{Z}$-basis of $T_X\subset\LL$, then $u^2,v^2\in 4\mathbb{Z}$ and $u\cdot v\in 2\mathbb{Z}$. Hence
$$\Tx\simeq\left(\begin{array}{cc} 4a & 2b \\ 2b &4c \end{array}\right)$$
where $a,b,c\in\mathbb{Z}$ satisfy the following conditions:
\begin{equation}\label{adm}\tag{$\star$}
\begin{cases}
d:=4ac-b^2>0;\\
b^2\leq ac\leq\textstyle\frac{d}{3};\\
-a\leq b\leq a\leq c.
\end{cases}
\end{equation}

Finally, observe that $X$ contains an ample class $L\in\pic$ which is $M_{20}$-invariant and such that $L^2=4n$ for some $n$ in $\mathbb{Z}_{\geq 1}$. One of our main goals is to describe the embedding $\langle 4n\rangle\hookrightarrow\LL$ in order to construct a projective model of $X$. To do this, the following result  is a direct consequence of \cite[Section 4.1]{SD}.

\begin{prop}\label{L/2+1}
If the linear system $\vert L\vert$ is not hyperelliptic, then it defines a map: $\varphi_L:X\rightarrow\mathbb{P}^{p_a(L)}$, where $p_a(L)=\frac{1}{2}L^2+1$.
\end{prop}

As we will see in Section \ref{projective_model}, the linear system $\vert L\vert$ is never hyperelliptic. As $L^2=4n$, the projective models of the K3 surfaces that we consider will be in $\mathbb{P}^{2n+1}$. 

The cases $n=1,2,10$ were already studied by C. Bonnafé and A. Sarti \cite{BS}, S. Brandhorst and K. Hashimoto \cite{BH} and S. Kond\=o \cite{Kondo}, when the surface $X$ admits the symplectic action of a maximal group $\Gamma$ containing $M_{20}$ properly. We recall their results and notations, as well as explicit equations for the K3 surfaces, in the following proposition. Trascendental lattices of these three surfaces are recalled in Appendix \ref{tab:T_X}. 
We use the standard notation $E_{\tau}$ for the elliptic curve $\mathbb{C}/(\mathbb{Z}+\tau\mathbb{Z})$ with $\tau\in\mathbb C$.

\begin{prop}{(see \cite{BS,BH,Kondo})}\label{3K3}
Let $G$ be a maximal group acting faithfully on a K3 surface $X$ which strictly contain the Mathieu group $M_{20}$. Recall that $L\in\pic$ is an ample class which is $M_{20}$-invariant. Thus there only are three possible cases:
\begin{enumerate}[1.]
    \item  $|G|=3840$ and $L^2=40$: the K3 surface $X_{\mathrm{Ko}}$ is the minimal resolution of the quotient of the Fermat quartic $\left\lbrace x^4+y^4+z^4+t^4=0\right\rbrace$ in $\mathbb P^3$ by the symplectic involution $(x:y:z:t)\mapsto (x:y:-z:-t)$. It corresponds to the Kummer surface $\mathrm{Km}\left(E_i\times E_i\right)$ (see \cite{Kondo}, \cite[Section 3]{BS}). 
    
    \item  $|G|=1920$ and $L^2=4$: the K3 surface $X_{\mathrm{Mu}}$ is defined as the following quartic in $\mathbb P^3$:
    \begin{equation}
    \label{eq-mukai}    
     \{x^4+y^4+z^4+t^4-6\left(x^2y^2+x^2z^2+x^2t^2+y^2z^2+y^2t^2+z^2t^2\right)=0\}.\end{equation}
    It corresponds to the Kummer surface $\mathrm{Km}\left(E_{i\sqrt{10}}\times E_{i\sqrt{10}}\right)$ (see \cite[p. 190]{Mukai} and \cite[Section 4]{BS}).
    \item  $|G|=1920$ and $L^2=8$: the K3 surface $X_{\mathrm{BH}}$ is defined as the complete intersection of the three quadrics in $\mathbb{P}^5$:
\begin{equation*}\label{BH}
\begin{cases}
x_0^2+x_3^2-\phi x_4^2+\phi x_5^2=0\\
x_1^2-\phi x_3^2+x_4^2-\phi x_5^2=0\\
x_2^2+\phi x_3^2-\phi x_4^2+x_5^2=0
\end{cases}\end{equation*}
where $\phi=\frac{1+\sqrt5}{2}$ (see \cite[Section 5]{BS}).
\end{enumerate}
\end{prop}

 \begin{rmk}
We observe that, according to \cite[Corollary 7.3]{D}, up to projective transformation there exists a unique octic K3 surface $X_8\subset\mathbb P^5$ which admits a faithful symplectic action of the Mukai group $M_{20}$.\end{rmk}

\section{The case $L^2=12$}\label{L^2=12}

We now start by studying the first new case, i.e. the polarization with $L^2=4n$ with $n=3$.

\subsection{Existence of the projective model}
Suppose that $L^2=12$. We denote by $(e,f,h)$ the standard basis of $\LL$, that is the vectors in $\LL$ which give the matrix of the bilinear form written before. We search $\lambda ,\mu ,\delta\in\mathbb{Z}$ such that
$$L=\lambda e+\mu f+\delta h.$$
Thanks to \cite[Lemma 2.8]{BS} we have
\begin{equation}\label{eq_L^2}
L^2=(2\lambda -\delta)^2+(2\mu -\delta)^2+10\delta^2 .
\end{equation}
If $L^2=12$, then possible solutions for $\delta$ are $\delta\in\lbrace -1,0,1\rbrace$ and
\begin{itemize}
\item if $\delta =0$, then $2\lambda^2+2\mu^2=12$, that is $\lambda^2+\mu^2=6$. This is not possible since $6$ is not the sum of two squares;
\item if $\delta=1$, then $(2\lambda-1)^2+(2\mu-1)^2=2$, that is $(2\lambda-1)^2=(2\mu-1)^2=1$. Hence $\lambda ,\mu\in\lbrace 0,1\rbrace$ and then
$$L\in\lbrace h;e+h;f+h;e+f+h\rbrace;$$
\item if $\delta=-1$, as before one has $\lambda ,\mu\in\lbrace 0,-1\rbrace$, then
$$L\in\lbrace -h;-e-h;-f-h;-e-f-h\rbrace.$$
\end{itemize}

Hence there are 8 possibilities for $L$ and they are all in the same orbit, up to isometry. In fact observe that the possibilities for $\delta=-1$ are the same as in the case $\delta=1$ up to the isometry $-\mathrm{id}_{\LL}$. Moreover we have:
\[-\mathrm{id}\circ\rho_2(e+f+h) =f+h, \quad 
\rho_1(f+h) = e+h, \quad
-\mathrm{id}\circ\rho_2(e+h) = h. \]

Therefore, up to isometry, there exists a unique embedding of lattices $\langle 12\rangle\hookrightarrow\LL$ which maps $L$ on $h$. 

We assume now that $L=h$. As $\Tx =L^{\perp}\cap\LL$, the orthogonal complement of the lattice $\mathbb{Z}h$ in $\LL$ is given by vectors $\lambda e+\mu f+\delta h$ such that
$$\langle \lambda e+\mu f+\delta h , h\rangle =0,\quad  \lambda ,\mu ,\delta\in\mathbb{Z}.$$

One can choose $u_1=f-e$ and $u_2=3e+3f+h$. Then $u_1$ and $u_2$ span the transcendental lattice $\Tx$.
Moreover $u_1^2=8$, $u_2^2=60$ and $\langle u_1,u_2\rangle =0$, so that $$\Tx = \left(\begin{array}{cc} 8 & 0 \\ 0 & 60 \end{array}\right)\, .$$ 
Using the notations introduced before, one has $a=2$, $b=0$ and $c=15$ and these integers satisfy the conditions \eqref{adm}, hence this is the matrix we are looking for.

\subsection{Construction of the K3 surface}
\label{modelX12}
The aim of this section is to find the K3 surface in $\mathbb{P}^7$ with an action of $M_{20}$. To understand the problem, we have the following diagram:

\begin{displaymath}
\xymatrix{
& & & M_{20}\ar@{-->}[ld] \ar@{^{(}->}[d] & \\
1\ar[r] & \mathbb{C}^{\times}\ar[r] & \mathrm{GL}_8(\mathbb{C})\ar[r] & \mathrm{PGL}_8(\mathbb{C})\ar[r] & 1 }
\end{displaymath}

In general, the arrow $M_{20}\longrightarrow\mathrm{GL}_8(\mathbb{C})$ does not exist. So we must consider a group $\tilde{M_{20}}\subset\mathrm{GL}_8(\mathbb{C})$ such that:

\begin{itemize}
\item $Z(\tilde{M_{20}})\simeq\mu_d$;
\item $\quotient{\tilde{M_{20}}}{\mu_d}\simeq M_{20}$.
\end{itemize}

Here $\mu_d$ denotes the group of primitive $d$-th roots of unity and $Z(\tilde{M_{20}})$ is the center of the group $\tilde{M_{20}}$ in $\mathrm{GL}_8(\mathbb{C})$. Then we have the following diagram:
\begin{displaymath}
\xymatrix{
1\ar[r] & \mu_d\ar[d]\ar[r] & \tilde{M_{20}}\ar[d]\ar[r] & M_{20}\ar@{^{(}->}[d]\ar[r] & 1 \\
1\ar[r] & \mathbb{C}^{\times}\ar[r] & \mathrm{GL}_8(\mathbb{C})\ar[r] & \mathrm{PGL}_8(\mathbb{C})\ar[r] & 1 }
\end{displaymath}
We want to understand the non-trivial central extensions of $M_{20}$ by a cyclic group and, thanks to the web version of the Atlas of finite groups \cite{atlas}, we know there are $6$ such extensions:
\begin{enumerate}
\item[•]$M_{20}$;
\item[•]$H_1$, $H_2$ and $H_3$, with $Z\left(H_k\right)\simeq\mu_2$, $k=1,2,3$;
\item[•]$G_1$ and $G_2$,  with $Z\left(G_k\right)\simeq\mu_4$, $k=1,2$.
\end{enumerate}

Denote by $G$ one of these groups. We are looking for the irreducible representations of $G$ in $\mathrm{GL}_8(\mathbb{C})$. In our case, by using the computer algebra system MAGMA \cite{MAGMA}, we obtain that only the group $G_1$ gives such a representation. Let $\rho:G\rightarrow\mathrm{GL}_8(\mathbb{C})$ be this representation. By abuse of notation, we identify $G$ with the image $\rho(G)$. The following holds:
$$Z(G)\simeq\mu_4 \quad\mathrm{and}\quad \quotient{G}{Z(G)}\simeq M_{20}\, .$$

Let $X$ be the K3 surface and assume for the moment that $X$ is defined by quadrics\footnote{In Section \ref{quadrics} we will show that the model is not hyperelliptic and $X$ is only defined by quadrics.} in $\mathbb{P}^7$. Denote $F:=\mathbb{C}\left[x_1,\dots ,x_8\right]_2$ the vector space of homogeneous polynomials of degree $2$ in $8$ variables and let $f_1,\dots ,f_n\in F$ be the equations of the quadrics defining $X$. 
Then $\mathbb{C}f_1+\dots +\mathbb{C}f_n\subset F$ must be stable by the action of $G$. We want to find the stable subspaces of $F$. By the theory of linear representations \cite[Proposition 8]{serre}, we have the following decomposition:
$$F = \bigoplus\limits_{S\in\mathrm{Irr}(G)}\mathrm{Im}\left(p_S^F\right)$$
where $p_S^F\in\mathrm{end}_{\mathbb{C}}(F)$ is the projection on the isotypic component associated with $S$.
Thanks to MAGMA, we obtain $F=F_6\oplus F_{10}\oplus F_{20}$, where the $F_i$ are $G$-stable, irreducible and of dimension $i$. Let $X_k$ be the subvariety in $\mathbb{P}^7$ defined by the elements of $F_k$. By using MAGMA, we can compute the dimension of each space: $\dim\left(X_6\right)=4$, $\dim\left(X_{10}\right)=2$ and $\dim\left(X_{20}\right)=-1$. This means that $X_{20}=\emptyset$ and $X_{10}$ is a surface, thus it is a good candidate for the surface we are looking for. Thanks to standard commands of MAGMA, we check the following conditions:
\begin{enumerate}
    \item[•] the canonical divisor is trivial;
    \item[•] the surface is smooth.
\end{enumerate}

At this point, it remains two possibilities: either $X_{10}$ is an abelian surface or it is a K3 surface. To conclude, we will compute the Euler characteristic $\chi(X_{10})$ of $X_{10}$ and therefore distinguish if $\chi(X_{10})=0$ and $X_{10}$ is abelian or $\chi(X_{10})=24$ and $X_{10}$ is a K3 surface. We consider
$$\mathcal{L} : g\in M_{20} \longmapsto \sum\limits_{i\geq 0} (-1)^i\mathrm{Tr}\left(g^{\ast}\vert\mathrm{H}^i\left(X_{10}\right)\right).$$

On the one hand, by the Lefschetz fixed point formula \cite[Chapter 3.4]{GH}, we have $\mathcal{L}(1)=\chi\left(X_{10}\right)$. On the other hand, we can compute $\mathcal L(1)$ using the following lemma:

\begin{lemma}\label{repr}
Let $\Gamma$ be a finite group and $f:\Gamma\rightarrow\mathbb{C}$ be a function which is the difference of two characters of representation. Then
$$f(1)\equiv -\sum\limits_{\underset{\gamma\neq 1}{\gamma\in\Gamma}} f(\gamma)\mod\vert\Gamma\vert\, .$$
\end{lemma}

Before proving this result, we need to recall some definitions and results coming from representation theory \cite[Chapters 1 and 2]{serre}.\\
Let $\left(E,\rho_E\right)$ be a linear representation of the group $\Gamma$. This means that $E$ is a complex vector space and $\rho_E:\Gamma\rightarrow\mathrm{GL}(E)$ is a group morphism. In particular, $\rho_{\mathbb{C}}$ is the so-called {\it trivial representation}, that is for all $\gamma\in\Gamma$, $\rho_{\mathbb{C}}(\gamma)=\mathrm{Id}_{\mathbb{C}}$. We call a {\it character} of the representation $\rho_E$ the morphism $\chi_E:\Gamma\rightarrow\mathbb{C}$ defined by 
\[\chi_E(\gamma)=\mathrm{Tr}\ \rho_E(\gamma), \ \gamma\in\Gamma.\] 
Finally,
we denote by $p_S^E$ the projection on the isotypic component associated with $S$. If $S$ is an irreducible representation of $\Gamma$, we have a characterization of this endomorphism which is
    $$p_S^E=\frac{\dim S}{\vert\Gamma\vert}\sum\limits_{\gamma\in\Gamma}\overline{\chi_S(\gamma)}\rho_E(\gamma).$$
\begin{proof}[Proof of Lemma \ref{repr}]
Let $S=\mathbb{C}$. For the endomorphism $p_\mathbb C^E$ we have
$$p^E_{\mathbb{C}}=\frac{1}{\vert\Gamma\vert}\sum\limits_{\gamma\in\Gamma}\overline{\chi_{\mathbb{C}}(\gamma)}\rho_E(\gamma).$$
As $\rho_{\mathbb{C}}$ is the trivial representation of $\Gamma$, for all $\gamma\in\Gamma$ we have $\chi_\mathbb{C}(\gamma)=\mathrm{Tr}\left(\mathrm{Id_{\mathbb{C}}}\right)$ and then
$$p^E_\mathbb{C}=\frac{1}{\vert\Gamma\vert}\sum\limits_{\gamma\in\Gamma}\rho_E(\gamma).$$
Taking the trace of this expression, with $f=\chi_E:\Gamma\rightarrow \mathbb C$ and multiplying it by $\vert\Gamma\vert$ we obtain the following:
$$f(1)=\vert\Gamma\vert\mathrm{Tr}\left(p_{\mathbb{C}}^E\right)-\sum\limits_{\underset{\gamma\neq 1}{\gamma\in\Gamma}}f(\gamma).$$
 
As $\mathrm{Tr}\left(p_{\mathbb{C}}^E\right)=\dim\mathrm{Im}\left(p_{\mathbb{C}}^E\right)$ is a positive integer, we can look at this expression modulo $|\Gamma|$. 
Similarly, if $f:\Gamma\rightarrow\mathbb{C}$ is as in the assumption,  we obtain the thesis.

\end{proof}

Let $X_{10}^g$ be the subspace of $X_{10}$ where $g\in M_{20}$ acts as the identity and observe that we have only isolated fixed points on it. Then, thanks again to the Lefschetz formula, for $g\neq 1$ we have $\mathcal{L}(g)=\chi\left(X_{10}\right)=\left\vert X_{10}^g\right\vert$. We compute this last term with MAGMA and by Lemma \ref{repr} we obtain
$$\mathcal{L}(1)\equiv 24\mod 960\, .$$

The surface $X_{10}$ has Euler characteristic equal to $24$, hence it is a K3 surface. Finally we get the equations for the quadrics which define the K3 surface, see Appendix \ref{app-A}.

\section{Other cases}\label{other_cases}
In order to repeat the construction for other values of $L^2$, we first observe for which $n$ the equation \eqref{eq_L^2} admits solution.

\begin{prop}\label{prop_no_existence}
The equation
\begin{equation}\label{2eq_L^2}
4n=(2\lambda -\delta)^2+(2\mu -\delta)^2+10\delta^2 \, .
\end{equation}
admits a solution $(\lambda,\mu,\delta)$ if and only if $n$ can not be expressed as $4^i(16j+6)$ for some non-negative integers $i,j$. 
\end{prop}

\begin{proof}
By Ramanujan ternary quadratic form, an even positive integer can be expressed as $x^2+y^2+10z^2$
for some $x,y,z\in\mathbb Z$ if and only if it is not of the form $4^i(16j+6)$ for some non-negative $i,j$, see \cite{Ram}. Then if $n\neq 4^i(16j+6)$, there exist $x,y,z\in\mathbb Z$ such that
\begin{equation}\label{eq-ram}
    4n=x^2+y^2+10z^2.
\end{equation}
Thus we can compute $(\lambda,\mu,\delta)$ satisfying \eqref{2eq_L^2} from the triple $(x,y,z)$:
$$\lambda=\frac{\pm x\pm z}2,\ \mu=\frac{\pm y\pm z}{2},\ \delta=\pm z.$$

Observe that $\lambda,\mu\in\mathbb Z$. 
In fact, if $z$ is even, then by \eqref{eq-ram} the sum $x^2+y^2$ has to be congruent to 0 modulo 4.
Since 3 is not a square modulo 4, then necessarily \[x^2\equiv 0\mod 4, \quad y^2\equiv 0\mod 4.\] This implies that $x^2$ and $y^2$ are even, so $x$ and $y$ are even too. Thus $\lambda,\mu\in\mathbb Z$. Similarly, if $z$ is odd, then $z^2\equiv 1 \mod 4$. Thus by \eqref{eq-ram} \[x^2\equiv 1\mod 4, \quad y^2\equiv 1\mod 4,\]  $x$ and $y$ are odd too and $\lambda,\mu\in\mathbb Z$. 
\end{proof}

Computations similar to the ones of Section \ref{L^2=12} allow to study the existence of the K3 surface with $L^2=4n$ for any value of $n$ (except for cases of Proposition \ref{prop_no_existence}). We show the results for $3\leq n\leq 10$, in order to give examples in several cases. We will see that some values  admit just an embedding, whereas for other values there are more embeddings and we will show why they are not in the same orbit. 

In Table \ref{tab:L^2}, we compute the values of $\lambda$, $\mu$ and $\delta$. To get all the possibilities, one can act on the given $(\lambda,\mu,\delta)$ with the elements of the isometry group $\mathcal O(\mathbb L_{20})$, i.e exchange $\lambda$ and $\mu$ and multiply the vector by $-1$.

\begin{table}[h]
    \centering
    \begin{tabular}{|c|c|cc|c|c|c|cc|cc|}
    \hline
         $L^2$ & $16$ & \multicolumn{2}{c|}{$20$} & $24$ & $28$ & $32$ & \multicolumn{2}{c|}{$36$} & \multicolumn{2}{c|}{$40$} \\
    \hline\hline
        $\lambda$ & 2 & 1 & $-1,2$ & - & 2 &2 & 3 & $-2,3$ & $3$ & $1,-1$ \\
    \hline
        $\mu$ & 0 & 2 & $0,1$ & - & -1 & 2 & 0 & $0,1$ & 1 & $1,-1$ \\
    \hline
        $\delta$& 0 & 0 & 1 & - & 1 & 0 & 0 & 1 & 0 & 2  \\
    \hline
    \end{tabular}
    \caption{Cases up to $L^2=40$}
    \label{tab:L^2}
\end{table}
  \begin{rmk}\label{diff_delta}
    Two vectors $(\lambda,\mu,\delta)$ and $(\lambda',\mu',\delta')$ with $\delta\ne\pm\delta'$ are not in the same orbit by the action of $\mathcal O(\mathbb L_{20})$.
    This follows from the fact that $\rho_1(\lambda,\mu,\delta)=(\mu,\lambda,\delta)$ and both
    $\rho_2$ and $-id_{\LL}$ change the sign of $\delta$.
    
    Observe that the converse is not true, i.e. there exists vectors $(\lambda,\mu,\delta), (\lambda',\mu',\delta')$ such that $\delta=\delta'$ but not in the same orbit by the action of $\mathcal O(\LL)$.
    For example for $L^2=300$, the two vectors $(\lambda,\mu,\delta)=(3,6,5)$ and $(\lambda',\mu',\delta')=(5,0,5)$ are not in the same orbit, but $\delta=\delta'$.
    \end{rmk}

\subsubsection*{$L^2=16$} There are $4$ possibilities for $(\lambda,\mu,\delta)$ and they are in the same orbit modulo isometry. 
Hence there exists a unique embedding $\langle 16\rangle\hookrightarrow\LL$ which maps $L$ to $2e$. The lattice $\Tx$ obtained in this case is the same as the case $L^2=4$ and we will see in Section \ref{prim} the relation between these two examples.
    
\subsubsection*{$L^2=20$} Up to isometry, we obtained two embedding $\langle 20\rangle\hookrightarrow\LL$:
    $$L\longmapsto e+2f \quad\mathrm{and}\quad L\longmapsto f-h\, .$$
    These embeddings are not in the same orbit, due to Remark \ref{diff_delta}. 
  
   Let us consider the first embedding $L\longmapsto e+2f$; we can choose $u_1=f-2e$ and $u_2=e+f+2h$ as generators of the transcendental lattice and get 
    $$\Tx=\left(\begin{array}{cc} 20 & 0 \\ 0 & 40 \end{array}\right).$$
    In the second case with $L\longmapsto f-h$, we take
    $u_1=2f+e+h$ and $u_2=f-3e$ and get
    $$\Tx=\left(\begin{array}{cc} 20 & 0 \\ 0 & 40 \end{array}\right).$$
    For both cases the conditions \eqref{adm} hold.
 \begin{rmk} 
We observe that for $L^2=20$ there are two non isometric vectors $(\lambda,\mu,\delta)$ but the matrix for the lattice $T_X$ is the same, therefore the same K3 surface admits two different actions of $M_{20}$.
   \end{rmk}
\subsubsection*{$L^2=24$} In this case there is no embedding by Proposition \ref{prop_no_existence}.

\subsubsection*{$L^2=28$} In this case $\delta=\pm1$. If $\delta=1$, then $\lambda,\mu\in\{2,-1\}$ and if $\delta=-1$, then $\lambda,\mu\in\{-2,1\}$. This gives 8 possibilities for $(\lambda,\mu,\delta)$, all in the same orbit. Thus we can choose $L\mapsto e+f-h$. Taking $u_1=f-e$ and $u_2=4e+4f+3h$  we compute
    $$\Tx=\left(\begin{array}{cc} 8 & 0 \\ 0 & 140 \end{array}\right)$$
    and the conditions \eqref{adm} hold for this matrix.
    \subsubsection*{$L^2=32$} In this case one obtains the same $\Tx$ as the case $L^2=8$. This case is explained in Section \ref{prim}.
    \subsubsection*{$L^2=36$} In this case, there are two possibilities:
    $$L\mapsto 3e \quad\mathrm{and}\quad L\mapsto 3e+h\, .$$
    
    The first embedding which maps $L$ on $3e$ gives the same transcendental lattice as $L^2=4$ because $36=3^2\cdot 4$ (see Section \ref{prim}). 
    
    For the second embedding, one can consider $u_1=3f+h$ and $u_2=2f+e-h$. 
The matrix then is $$\Tx=\left(\begin{array}{cc} 36 & 12 \\ 12 & 44 \end{array}\right)$$
and the conditions \eqref{adm} hold for this matrix.    
    \subsubsection*{$L^2=40$} There are two possibilities, up to isometry. One is $L\mapsto e+f+2h$ and this one is studied in \cite{Kondo}.
    The other one is $L\mapsto 3f+e$.
    Taking $u_1=f-e+h$ and $u_2=2e+h$ one gets $u_1^2=u_2^2=20, u_1\cdot u_2=0$, i.e. 
    $$\Tx=\left(\begin{array}{cc} 20 & 0 \\ 0 & 20 \end{array}\right)$$
    and the conditions \eqref{adm} hold for this matrix.
    \begin{rmk} Observe that in \cite[Lemma 3.1, Prop. 3.3]{Kondo}, Kond\=o proves that the surface $X=Km(E_i\times E_i)$ is the unique surface admitting the action of a symplectic group $G$ such that $G_0\simeq M_{20}$ and $G/G_0\simeq \mathbb Z/4\mathbb Z$.
    In this case $T_X=\left(\begin{array}{cc} 4 & 0 \\ 0 & 4 \end{array}\right)$.
    The fact that we find here two different surfaces (with two different trascendental lattices) is not contradicting the unicity proven in \cite{Kondo}, since here we are not assuming the action of $G$ such that $G_0\simeq M_{20}$ and $G/G_0\simeq \mathbb Z/4\mathbb Z$.
    \end{rmk}
    
Let $X_{\mathrm{Ko}}$ be the surface with $T_{X_{\mathrm{Ko}}}=\left(\begin{array}{cc} 4 & 0 \\ 0 & 4 \end{array}\right)$
according to the notation of \cite[Section 3]{BS}, 
and let $X'$ be the K3 surface with transcendental lattice  $T_{X'}=\left(\begin{array}{cc} 20 & 0 \\ 0 & 20 \end{array}\right)$. 
We can observe that  $T_{X'} =5 T_{X_{\mathrm{Ko}}}$ and this relates to the fact that 
by \cite[Section 5]{Sm}, $X'$ is the Kummer surface $\mathrm{Km}\left(E_{5i}\times E_{5i}\right)$. 

\section{About the projective models}\label{projective_model}

In order to describe the projective models of K3 surfaces with an action of $M_{20}$, we want to understand when the linear system $\vert L\vert$ is hyperelliptic, i.e. when there exists a surface $S\subset\mathbb{P}^n$ such that the map $X\rightarrow S$ given by $\vert L\vert$ is $2:1$, and when $L$ defines an embedding. With the previous notations, we will show the following:

\begin{thm}\label{mainTHM}
Given the K3 surface $X$ with the embedding given by $L$ and $L^2=4n$, we have the following results:
\begin{enumerate}[1.]
    \item $L$ is an ample class;
    \item the linear system $\vert L\vert$ is not hyperelliptic;
    \item if $L^2\geq 8$, then the projective model is only defined by quadrics.
\end{enumerate}
\end{thm}

Before we start proving the Theorem, we introduce the following two Lemmas which will be very useful for the following Sections.

\begin{lemma}\label{index1}
Let $L$ be a primitive ample class which is $M_{20}$-invariant on the K3 surface $X$. We assume that there exists a curve $C$ such that $C\cdot L=k\neq 0$. Then $4n$ divides $kI$, where $I:=\left[\mathrm{NS}(X):\mathbb{Z}L\oplus\LL^{\perp}\right]$, and $n$ divides $10k^2$.
\end{lemma}

\begin{proof}
    First we will show that $4n$ divides $kI$. Note that $(k\Z+4n\Z)/4n\Z$ is a subquotient of $\pic/\left(\Z L\oplus\LL^{\perp}\right)$. In fact the map
    $$\begin{array}{cccl}
        \pic & \longrightarrow & \pic\cdot L & \subset\Z \\
         v & \longmapsto & v\cdot L &
    \end{array}$$
    induces a surjection 
    $$\quotient{\pic}{\Z L\oplus\LL^{\perp}}\twoheadrightarrow\quotient{\pic\cdot L}{4n\Z}\, .$$
    By assumption there exists a curve $C$ such that $L\cdot C=k$, so $\pic\cdot L\supset k\Z+4n\Z$. Thus $I':=\left[k\Z+4n\Z:4n\Z\right]$ divides $I$ and, since $I'=\frac{4n}{\gcd\{k,4n\}}\in\frac{4n}{k}\Z$, we obtain $I\in\frac{4n}{k}\Z$. Hence $4n$ divides $kI$. \\

    Now we can deduce that $n$ divides $10k^2$. By \cite[Chapter I, Lemma 2.1]{BHPV} we have
    $$I^2=\frac{\det(\Z L)\cdot\det\left(\LL^{\perp}\right)}{\det\pic}=\frac{640n}{\det\Tx}\, .$$
    Since $\frac{kI}{4n}$ is an integer, so is its square
    $$\frac{k^2I^2}{16n^2}=\frac{10k^2}{n\left(\det\Tx/4\right)}.$$

    Note that $\det\Tx=4\left(4ac-b^2\right)$ is always divisible by $4$, thus we get the Lemma.
\end{proof}

As a consequence we have the following result.

\begin{lemma}\label{index2}
    Under the assumptions of Lemma \ref{index1}, if we assume $k\in\{1,2,3\}$, then $(n,k)\in\{(1,2),(2,2),(10,2)\}$, and $X$ is one of the examples studied in \cite{BS} (see Proposition \ref{3K3}).
\end{lemma}

\begin{proof}
Suppose that $k\in\{1,2,3\}$,  we see what happens in each case.
\begin{itemize}
    \item\textbf{Case $k=1$}: by Lemma \ref{index1} $n$ divides $10$, i.e. $n\in\{1,2,5,10\}$. Moreover $4n$ should divide $I$, but if we look at Table \ref{tab_complete} in Appendix \ref{tab:T_X} we see that it is impossible for each of the four cases.
    \item\textbf{Case $k=2$}: by Lemma \ref{index1} $n$ divides $40$, i.e. $n\in\{1,2,4,5,8,10,20,40\}$. Moreover $4n$ should divide $2I$, and if we look at Table \ref{tab_complete} in Appendix \ref{tab:T_X} we see that it is possible only for $n=1$, $n=2$ and $n=10$ with $X_{\mathrm{Ko}}$.
    \item\textbf{Case $k=3$}: by Lemma \ref{index1} $n$ divides $90$, i.e. $n\in\{1,2,3,5,6,9,10,15,18,30,45,90\}$. Moreover $4n$ should divide $3I$, but if we look at Table \ref{tab_complete} in Appendix \ref{tab:T_X} we see that it is impossible.
\end{itemize}
\end{proof}

For the following, we assume that $L$ is $M_{20}$-invariant and $L^2=4n$, with $n\in\Z_{\geq 1}$.
The proof of Theorem \ref{mainTHM} will be done in the following four Sections.

\subsection{$L$ is ample}
First, observe that we can suppose $L$ to be an effective class. In fact by Riemann-Roch Theorem \cite[Section V, Theorem 1.6]{hart}, we have for any divisor $D$ on a K3 surface:

$$h^0(D)+h^0(-D)=2+\frac{D^2}{2}+h^1(D)$$

where $h^i(D)=\dim\mathrm{H}^i(D)$, for $i\geq 0$. In our case $L^2=4n>0$, so the sum $h^0(L)+h^0(-L)$ is strictly positive. By definition, $h^0(L)$ is strictly positive if and only if $L$ is an effective class. Hence $L$ or $-L$ is effective and so we can suppose $L$ to be effective.

Since $\mathrm{NS}(X)\supset\mathbb{Z}L\oplus\LL^{\perp}$, we can consider $C\in\mathrm{NS}(X)$ a $(-2)$-curve such that
$$C=\frac{\alpha L+v}{\eta}$$
where $\alpha$ and $\eta\neq 0$ are two integers and $v\in\LL^{\perp}$. Moreover $\alpha\neq\ 0$ because $\LL^{\perp}$ does not contain any $(-2)$ by the point b) of \cite[Lemma 4.2]{Nikulin}. 

We can be more precise about the integers $\alpha$ and $\eta$:
\begin{itemize}
    \item $\eta$ is positive because $\frac{L}{\eta}$ is in the discriminant group.
    \item $L$ is effective and $C$ is a curve, so $L\cdot C\geq 0$. Thus $\frac{\alpha}{\eta}\geq 0$ and $\alpha$ is also positive.
\end{itemize}

So we obtain that $\frac{\alpha}{\eta}>0$. Hence $L\cdot C>0$, which implies by Nakai-Moishezon criterion (see \cite[Chapter V, Theorem 1.10]{hart}) that $L$ is an ample class.

\subsection{$\vert L\vert$ has no fixed part}
By \cite[Section 3.8]{reid}, either $\vert L\vert$ has no fixed part or  $L=aE+\Gamma$, with $a$ a positive integer, $E$ an elliptic curve and $\Gamma$ an irreducible $(-2)$-curve such that $E\cdot\Gamma=1$. Suppose that we are in the second case. Then $E\cdot L=E\cdot\Gamma=1$ and by Lemma \ref{index2} we can conclude that the linear system $\vert L\vert$ has no fixed part.

\subsection{$\vert L\vert$ is not hyperelliptic} 
By \cite[Theorem 5.2]{SD} $L$ is hyperelliptic only in the two following cases:
\begin{itemize}
    \item There exists an irreducible curve $E$ of genus 1 such that $E\cdot L=2$.
    \item There exists an irreducible curve $B$ of genus 2 such that $L=2B$.
\end{itemize}

First, assume that there exists an irreducible curve $B$ of genus $2$ such that $L=2B$. Remark that $B^2=2$, so $L^2=4B^2=8$. However, we already know that in this case the model is not hyperelliptic by \cite{BS}.

Now suppose that there exists an irreducible curve $E\in\mathrm{NS}(X)$ of genus $1$ such that $E\cdot L=2$. By Lemma \ref{index2} we know that $n\in\{1,2,10\}$ and then $X$ corresponds to one of the three surfaces $X_\mathrm{Mu}$, $X_{\mathrm{BH}}$ or $X_{\mathrm{Ko}}$ already studied (see Proposition \ref{3K3}).

\begin{rmk}
We observe that by \cite[\S 4.1]{SD}, since $|L|$ is not hyperelliptic and has no fixed part, it is very ample.
\end{rmk}

\subsection{Quadrics on the surface}\label{quadrics}
We now prove the last part of Theorem \ref{mainTHM}, i.e. that assuming $L^2\geq 8$, the surface is  defined by an intersection of quadrics. By \cite[Theorem 7.2]{SD} this is true except in the two following cases:

\begin{itemize}
\item there exists an irreducible curve $E$ of genus 1 with $E\cdot L=3$, or
\item $L=2B + F$ where $B$ is an irreducible curve of genus 2 and $F$ is an irreducible rational curve such that $B\cdot F=1$.
\end{itemize}

By Lemma \ref{index2}, the first point is impossible. The second point is not possible either, since $\left(2B+F\right)^2=10\neq 4n$. \\

This concludes the proof of Theorem \ref{mainTHM}. We can be more precise and compute the exact number $Q_{L^2}$ of quadrics which define the K3 surface associated to $L$. 

\begin{prop}\label{number}
Let $X\subset\mathbb P^N$ with the embedding $L$ such that $L^2=4n\geq 8$. The number of quadrics defining $X$ in $\mathbb P^N$ is $$Q_{4n}=2n^2-3n+1.$$
\end{prop}

\begin{proof}
Suppose that we have $X\subset\mathbb P^N$ with the embedding determined by $L$ such that $L^2=4n$. We know that $N=2n+1$ by Proposition \ref{L/2+1}. Following \cite[Section 6.5.3]{SD}, one can compute the number of quadrics defining the surface as
 $$\dim S^2H^0(L)-\dim H^0(2L),$$
where $\dim S^2H^0(L)$ is the total number of quadrics in the projective space $\mathbb P^{2n+1}$,  that is $\binom{2n+3}{2}$. By Riemann Roch theorem \cite[Section V, Theorem 1.6]{hart} we can compute $\dim H^0(2L)=2+8n$. Finally, we conclude that the number of quadrics defining $X$ is
$$Q_{4n}:=\binom{2n+3}{2}-(2+8n)=2n^2-3n+1.$$
\end{proof}

\section{Polarizations which are not primitives}\label{prim}

With the computations of $T_X$ (see Appendix \ref{tab:T_X}), one could remark that there are similarities between some cases. For example, we get the same lattice $T_X$ when $L^2=4$ and $L^2=16$, as well as in cases $L^2=8$ and $L^2=32$. 
Also, there is a relation between these embeddings. Actually, we have the following result:

\begin{prop}\label{2L}
Given the embedding $\langle 4n\rangle\hookrightarrow\LL$, then, for all integers $r\geq 1$, there exists an embedding $\langle r^2 4n\rangle\hookrightarrow\LL$ and such that the lattice $\Tx$ obtained for $L$ is the same as the one for $rL$.
\end{prop}

\begin{proof}
For the existence of the embedding $\langle r^24n\rangle\hookrightarrow\LL$, we can assume that there exists $\lambda_0,\mu_0,\delta_0\in\mathbb{Z}$ such that $L\mapsto\lambda_0 e+\mu_0 f+\delta_0 h$. 
Hence $rL\mapsto \left(r\lambda_0\right) e+\left(r\mu_0\right) f+\left(r\delta_0\right) h$ is the desired embedding.
The transcendental lattice $\Tx$ is obviously the same. 
\end{proof}

The previous proposition allows to describe explicitly some non-primitive cases, i.e. cases where the polarization is multiple of some other polarization $L$. 

Throughout this section, $X_{4n}$ will denote the K3 surface admitting the polarization $L$ with $L^2=4n$. 
Thanks to Proposition \ref{quadrics}, we are able to compute the number of quadrics defining $X_{4n}$. In  cases which are not primitive, we can be more precise and give explicitly the equations of the quadrics defining the surface. To do this, we first recall the Veronese embedding of degree $d$, see \cite{harris}:

$$\begin{array}{cccc}
\nu_d^n : & \mathbb{P}^n & \longrightarrow & \mathbb{P}^m \\
 & \left(x_0:\dots :x_n\right) & \longmapsto & \left(x_0^d:x_1^d:\dots:x_n^d:x_0^{d-1}x_1:x_0^{d-1}x_2:\dots\right) \\
\end{array}\, .$$

Since we use monomials of degree $d$ in $n$ variables, $m=\binom{n+d}{d}-1$.

\subsection{Cases $L^2=4\cdot(4n)$}
We will first show explicitly some examples and then state the general property.

\begin{example}[Case $L^2=16$]
To study the case $L^2=16$, we first consider the surface $X_{\mathrm{Mu}}$ with the polarisation given by $M$, with $M^2=4$.  By Proposition \ref{3K3}, $X_{\mathrm{Mu}}$ is the zero locus of a quartic in $\mathbb P^3$ whose equation is given in \eqref{eq-mukai}. 

Now let $L=2M$ and we consider the Veronese embedding of degree 2:
$$\nu_2^3:\mathbb P^3\rightarrow \mathbb P^9.$$

Let $(y_0:\ldots:y_9)$ be the coordinates of $\mathbb P^9$. The image of the equation defining $X_{\rm Mu}$ in $\mathbb P^9$ via $\nu_2^3$ is given by
\begin{equation}
y_0^2+y_1^2+y_2^2+y_3^2-6(y_4^2+y_5^2+y_6^2+y_7^2+y_8^2+y_9^2)
\label{eq2}
\end{equation}

The image $\nu_2^3(\mathbb P^3)$ is defined by the quadrics given by the zero loci of the $2\times 2$ minors of the matrix
$$\mathcal{A}:=\left(\begin{array}{cccc}
     y_0& y_4&y_5&y_6 \\
     y_4& y_1&y_7&y_8 \\
     y_5& y_7&y_2&y_9 \\
     y_6& y_8&y_9&y_3 \\
\end{array}\right)\, .$$

This gives 20 equations of quadrics. Let us consider the K3 surface $X_{16}$ in $\mathbb P^9$ given by the 20 quadrics obtained as minors of the matrix $\mathcal{A}$ plus the quadric defined by \eqref{eq2}. The number of quadrics is $Q_{16}=21$, as desired.
As observed in Proposition \ref{2L}, the surface $X_{16}$ admits a polarization $L$ such that $L^2=4M^2=16$ and it inherits the action of $M_{20}$, as well as the action of $\mu_2$ described in \cite[Section 4]{BS}.\end{example}

\begin{example}[Case $L^2=64$]
We have previously shown a model of $X_{16}$ in $\mathbb P^9$ defined by $Q_{16}=21$ quadrics. In order to exhibit a projective model of a K3 surface $X_{64}$ with a polarization $L$ with $L^2=64$, we consider the Veronese embedding $\nu_2^{9}:\mathbb P^{9}\rightarrow \mathbb P^{54}$. The 21 quadrics defining $X_{16}\subseteq\mathbb P^9$ give 21 hyperplanes in $\mathbb P^{54}$, thus their intersection determines a 33-dimensional space. The desired surface $X_{64}\subset\mathbb P^{33}$ is given by the restriction of the quadrics defining $\nu_{2}^{9}(\mathbb P^9)$ to this 33-dimensional space.\end{example}

\begin{example}[Case $L^2=32$]
By \cite[Section 5]{BS} the K3 surface $X_{\rm BH}$ admits a polarization $M$ such that $M^2=8$. Taking $L=2M$, we obtain a polarization with $L^2=32$ on the same surface. We would like to describe its projective model which is in $\mathbb P^{17}$
by Proposition \ref{L/2+1}. 
By Proposition \ref{3K3}, case 3., $X_{\rm BH}$ is the intersection of 3 quadrics in $\mathbb P^5$.

In order to show the projective model of $X_{\rm BH}$ in $\mathbb P^{17}$ we consider the Veronese embedding 
$\nu_2^5:\mathbb P^5\rightarrow \mathbb P^{20}$. The image of $\mathbb P^5$ via $\nu_2^5$ is the zero locus of quadrics and their number is 
$$\frac12\left(\binom62^2+\binom62\right)-\binom{6}{4}=105\ .$$
The last term $\binom{6}{4}$ considers the Pl\"ucker relations.
The images of $q_1$, $q_2$ and $q_3$ via $\nu_2^5$ give three linear equations, thus $q_1,q_2,q_3$ define 3 hyperplanes in $\mathbb P^{20}$. 
Their intersection is a $17$-dimensional space and in this projective space of dimension 17 the equation of $X$ is given by the restrictions of the $105$ equations defining $\nu_2^5(\mathbb P^5)$.
\end{example}

In general let $M^2=4n$. Recall that the number of quadrics defining $X_{4n}$ by Proposition \ref{number} is 
$$Q_{4n}=2n^2-3n+1.$$

If we consider $L=2M$, therefore $L^2=4M^2=16n$ and we expect to find a projective model of $X_{16n}$ in $\mathbb P^{8n+1}$ by Proposition \ref{L/2+1}. We consider the Veronese embedding of degree 2
$$\nu_2^{2n+1}:\mathbb{P}^{2n+1}\hookrightarrow\mathbb{P}^{2n^2+5n+2}\, .$$

The $Q_{4n}$ quadrics defining $X_{4n}$ give via $\nu_2$ a set of independent hyperplanes in $\mathbb P^{2n^2+5n+2}$, whose number is $Q_{4n}=2n^2-3n+1$.  
Thus we obtain a subspace of dimension $$(2n^2+5n+2)-(2n^2-3n+1)=8n+1$$
as expected by Proposition \ref{L/2+1}. 
The surface $X_{16n}$ is given by the restriction of the quadrics defining $\nu_2(\mathbb P^{2n+1})$ to this space of dimension $8n+1$.

Observe that this argument works for any value of $M^2=4n$. We showed explicitly what happens when $M^2=4$, $M^2=16$ and $M^2=8$ in the previous examples.

\subsection{Cases $L^2=4r^2$} Another interesting case is when we take $L=rM$ with $M^2=4$.
Let $X_{\rm Mu}$ be the K3 surfaces defined in Proposition \ref{3K3}, with $M^2=4$, whose projective model is given by the zero locus of a polynomial $f_4$ of degree 4 in  $\mathbb P^3$. We show how to obtain the projective model of the surface with embedding $L=rM$, thus having $L^2=4r^2$.

\begin{example}[Case $L^2=36$, $r=3$] We expect a model of $X_{36}$ in $\mathbb P^{19}$ by Propositon \ref{L/2+1}. 
According to \cite[Example 2.4]{harris}, one can prove $X_{\rm Mu}=\{f_4=0\}$ can also be defined as the intersection
of the zero loci of \[q_2(x_0,\ldots,x_3)f_4(x_0,\ldots,x_3)=0\] where $q_2(x_0,\ldots,x_3)$ is one of the 10 elements of the basis of polynomials of degree 2 in $(x_0,\ldots,x_3)$, i.e.:
\[x_0^2,\ x_1^2,\ x_2^2,\ x_3^2,\ldots ,x_2x_3.\]
Thus $X_{\rm Mu}$ is defined as the intersection of the zero loci of 10 polynomials $g_1,\ldots,g_{10}$ of degree 6.
Let $\nu_3^3:\mathbb P^3\rightarrow \mathbb P^{19}$ be the Veronese embedding of degree 3. The image of each $g_i$ is a polynomial of degree 2 in $\mathbb P^{19}$.  We thus obtain $X_{36}$ as the intersection of the 10 quadrics in $\mathbb P^{19}$ given by the $g_i$'s and the quadrics defining the image of the Veronese embedding.
\end{example}

\begin{example}[Case $L^2=100$, $r=5$] If $r=5$ and we consider $L=5M$ with $M^2=4$, then $L^2=100$.
Let $f_4$ be the polynomial of degree 4 defining the surface $X_{\rm Mu}$. The zero locus $\{f_4=0\}$ can also be described by the intersection of the zero loci of
\begin{equation}\label{hyperplanes}
x_0f_4=x_1f_4=x_2f_4=x_3f_4=0. 
\end{equation}

Via the Veronese embedding $\nu^3_5:\mathbb P^3\rightarrow \mathbb P^{\binom85-1}=\mathbb P^{55}$ the surface $X_{\rm Mu}$ has as image a K3 surface admitting a polarisation $L=5M$. The equations \eqref{hyperplanes} define 4 hyperplanes in $\mathbb P^{55}$, thus the model of $X_{100}$ is contained in $\mathbb P^{51}$.
\end{example}

More generally, for any $r>3$, via the Veronese embedding $\nu^3_r:\mathbb P^3\rightarrow \mathbb P^{\binom{r+3}3-1}$ one obtains a K3 surface admitting a polarization $L=rM$ and thus $L^2=4r^2$. 
This surface is obtained as the image of $X_{\rm Mu}$ via $\nu^3_r$. As observed, the surface $X_{\rm Mu}$ can be described by the intersection
of the zero loci of
\begin{equation}\label{hyp_gen}
q_1f_4=\ldots=q_sf_4=0
\end{equation}

with $s=\binom{r-1}{3}$ and the $q_i$'s are monomials of degree $r-4$. Equations \eqref{hyp_gen} define $\binom{r-1}{3}$ hyperplanes in $\mathbb P^{\binom{r+3}3-1}$ thus one obtains a model for the K3 surface $X_{4r^2}$ in a space of dimension \[\binom{r+3}3-1-\binom{r-1}{3}=2r^2+1\]
as expected.

\section{Existence of infinitely many K3 surfaces with an action of $M_{20}$}\label{infinity}

It is known by \cite[Propostion 2.1]{BS} that there are at most countably many K3 surfaces with a symplectic and faithful action of $M_{20}$.
It is interesting to ask whether there exists an infinite number of such K3 surfaces.
According to previous Sections, one can consider the polarization $L^2=4p$ with $p$ a prime. If equation \eqref{eq_L^2} admits a solution for $L^2=4p$, one can repeat the arguments used before and by Theorem \ref{T_X} one obtains a K3 surface with the action of $M_{20}$. Moreover, since $p$ is a prime, this surface is a new one, meaning that it does not admit a polarization $L^2=4m$ with $m<p$. 

\begin{thm}\label{infinite}
There exist infinitely many K3 surfaces admitting the action of $M_{20}$.
\end{thm}
\begin{proof}
The proof will be done in two steps. First we will show that there exists infinitely many K3 surfaces with the previous assumptions. Then we will show that these surfaces admit an action of $M_{20}$ by using \cite[Lemma 8.24]{Kondobook}.

In order to prove that there exists infinitely many such K3 surfaces, it suffices to show that there are infinite number of primes $p$ such that the polarization $L^2=4p$ defines an embedding and thus a K3 surface. 
In order to answer this, we need to check that there is an infinite number of primes $p$ such that the equation 
\begin{equation}\label{eq_L^2_primes}
4p=(2\lambda -\delta)^2+(2\mu -\delta)^2+10\delta^2 .
\end{equation}
 admits a solution $(\lambda, \mu, \delta)$.

We observe that if $p\equiv 1\ (4)$, one can take $\delta=0$. The equation becomes
\[p=\lambda^2+\mu^2\]
and Fermat's theorem of two squares \cite[Chapter 1.1]{Primes} ensures the existence of two integers $\lambda, \mu$.

By the weak form of Dirichlet theorem, there exists infinitely many primes congruent to $1$ modulo $4$.
Then there exists infinitely many primes such that the equation \eqref{eq_L^2_primes} admits a solution $(\bar\lambda,\bar\mu,0)$. It follows that there is an infinite number of K3 surfaces with an action of $M_{20}$. \\

It remains to prove that the group $M_{20}$ acts on these surfaces. 
First we consider an embedding of $\LL$ in $\Lambda_{K3}$ that is isometric to the embedding of  $\mathrm{H}^2\left(X',\mathbb{Z}\right)^{M_{20}}$ in $\mathrm{H}^2\left(X',\mathbb{Z}\right)$ for some K3 surface $X'$ admitting an action of the group $M_{20}$. 

We take an element $L\in\LL$ such that $L^2=4p$ as above and a line $\mathbb{C}\omega\subset\LL\otimes_{\mathbb{Z}}\mathbb{C}$ that is orthogonal to $L$ and such that $\omega\cdot\omega=0$ and $\omega\cdot\bar{\omega}>0$. In this way $\omega$ can be viewed as an element of the period domain $\Omega_{\Lambda_{K3}}$. By the surjectivity of the period map of K3 surfaces (see for example \cite[Chapter 7]{Kondobook}), there exists a K3 surface $X$ and an ample line bundle $L_X$ on $X$ such that the isometry $\mathrm{H}^2(X,\mathbb{Z})$ allows us to identify $L_X$ to $L$ and the 1-dimensional subset $\mathrm{H}^0\left(X,\Omega^2_X\right)\subset\mathrm{H}^2(X,\mathbb{C})$ to $\mathbb{C}\omega$.

We can now apply \cite[Lemma 8.24]{Kondobook} to prove that $M_{20}$ acts on $\mathrm{H}^2(X,\mathbb{Z})$. Condition (1) of \cite[Lemma 8.24]{Kondobook} states that the group $M_{20}$ has to preserve $\mathrm{H}^0\left(X,\Omega^2_X\right)$ and this the case by construction. For condition (2)  the class $L_X$ 
clearly satisfies the conditions. 
Finally, condition (3) is also satisfied since $\LL^{\perp}$ contains no element of norm $-2$ by \cite[Lemma 4.2]{Nikulin}. As a consequence the Mathieu group $M_{20}$ acts on $X$ with $L_X\in\mathrm{H}^2(X,\mathbb{Z})^{M_{20}}$ as desired.
\end{proof}

\newpage
\appendix
\section{Table of $T_X$}\label{tab:T_X}

We recall the notations of Table \ref{tab_complete}. Let $X$ be a K3 surface admitting a faithful and symplectic action of the Mathieu group $M_{20}$ and let $T_X$ be its transcendental lattice, with $T_X=\left(\begin{array}{cc}
     4a & 2b  \\
     2b & 4c 
\end{array}\right)$ and $a,b,c$ as in Section \ref{sec-background}, satisfying \eqref{adm}. We denote by $L\in\mathrm{NS}(X)$ a $M_{20}$-invariant ample class such that $L^2=4n$, with $n\in\mathbb{Z}_{\geq 1}$ and by $Q_{L^2}$ the number of quadrics that describe the projective model, according to Proposition \ref{number}. 
Finally, $I=\sqrt{\frac{160n}{4ac-b^2}}$ is the index of $\mathbb{Z}L\oplus\LL^{\perp}$ in $\mathrm{NS}(X)$.
When $n=1,2,10$, the surfaces are described in Proposition \ref{3K3}.

 \small
    \begin{longtable}[c]{| c| c | c | c|cc c| c | c| }
    \caption{Table with all cases  \label{tab_complete}}\\
    \hline
    $n$ & $L^2$ & $Q_{L^2}$ &$T_X$ & $a$ & $b$ & $c$ & $\langle 4n\rangle\hookrightarrow\LL$ &  $I=\sqrt{\frac{160n}{4ac-b^2}}$\\
        \hline

    \hline
    \endfirsthead
        \hline

    $n$ & $L^2$ & $Q_{L^2}$ &$T_X$ & $a$ & $b$ & $c$ & $\langle 4n\rangle\hookrightarrow\LL$ &
    $I=\sqrt{\frac{160n}{4ac-b^2}}$\\
        \hline
    \hline
     \endhead
    \hline
    1 & 4   &-& $\left(\begin{array}{cc} 4 & 0 \\ 0 & 40 \end{array}\right)$ & 1 & 0 & 10 & $L\mapsto e$ &  2 \\
    \hline
    2 & 8   &3& $\left(\begin{array}{cc} 8 & 4 \\ 4 & 12 \end{array}\right)$ & 2 & 2 & 3 & $L\mapsto e+f$ & 4 \\
    \hline
    3 & 12  &10& $\left(\begin{array}{cc} 8 & 0 \\ 0 & 60 \end{array}\right)$ & 2 & 0 & 15 & $L\mapsto h$ & 2 \\
    \hline
    4 & 16 &21 & $\left(\begin{array}{cc} 4 & 0 \\ 0 & 40 \end{array}\right)$ & 1 & 0 & 10 & $L\mapsto 2e$ & 4 \\
    \hline
    
    \multirow{2}{*}{5} &  \multirow{2}{*}{20}  & \multirow{2}{*}{36} & \multirow{2}{*}{$\left(\begin{array}{cc} 20 & 0 \\ 0 & 40 \end{array}\right)$ }&  \multirow{2}{*}{5} &  \multirow{2}{*}{0} &  \multirow{2}{*}{10} & $L\mapsto e+2f$ &    \multirow{2}{*}{2} \\
    & & && & & & $L\mapsto f-h$ & \\
    \hline
    6 & 24  & - & - & - & - & - & - & - \\
    \hline
    
    7 & 28 &80 & $\left(\begin{array}{cc} 8 & 0 \\ 0 & 140 \end{array}\right)$ & 2 & 0 & 35 & $L\mapsto e+f-h$ &  4 \\
    \hline
    8 & 32&105  & $\left(\begin{array}{cc} 8 & 4 \\ 4 & 12 \end{array}\right)$ & 2 & 2 & 3 & $L\mapsto 2e+2f$ &  8 \\
    \hline
     \multirow{3}{*}{9} &  \multirow{3}{*}{36}  &\multirow{3}{*}{136} & $\left(\begin{array}{cc} 4 & 0 \\ 0 & 40 \end{array}\right)$ & 1 & 0 & 10 & $L\mapsto 3e$ &   6 \\
    & && $\left(\begin{array}{cc} 36 & 12 \\ 12 & 44 \end{array}\right)$ & 9 & 6 & 11 & $L\mapsto 3e+h$ & 2\\
    \hline
     \multirow{3}{*}{10} &  \multirow{3}{*}{40}  &\multirow{3}{*}{171} & $\left(\begin{array}{cc} 4 & 0 \\ 0 & 4 \end{array}\right)$ & 1 & 0 & 1 & $L\mapsto e+f+2h$ &   20 \\
    & & & $\left(\begin{array}{cc} 20 & 0 \\ 0 & 20 \end{array}\right)$ & 5 & 0 & 5 & $L\mapsto e+3f$ &  4\\
    \hline
     \multirow{3}{*}{15} &  \multirow{3}{*}{60}  &\multirow{3}{*}{406} & $\left(\begin{array}{cc} 8 & 0 \\ 0 & 12 \end{array}\right)$ & 2 & 0 & 3 & $L\mapsto 2e+2f-h$ &   5 \\
    & & & $\left(\begin{array}{cc} 20 & 0 \\ 0 & 120 \end{array}\right)$ & 5 & 0 & 25 & $L\mapsto e-2h$ & 1 \\
    \hline
      \multirow{3}{*}{18} &  \multirow{3}{*}{72}  &\multirow{3}{*}{595} & $\left(\begin{array}{cc} 8 & 4 \\ 4 & 12 \end{array}\right)$ & 2 & 2 & 3 & $L\mapsto 3e+3f$ &   12 \\
    & & & $\left(\begin{array}{cc} 8 & 4 \\ 4 & 92 \end{array}\right)$ & 2 & 2 & 23 & $L\mapsto 3e+3f+2h$ & 4 \\

    \hline
    
    \multirow{2}{*}{20} &  \multirow{2}{*}{80}  & \multirow{2}{*}{741} & \multirow{2}{*}{$\left(\begin{array}{cc} 20 & 0 \\ 0 & 40 \end{array}\right)$ }&  \multirow{2}{*}{5} &  \multirow{2}{*}{0} &  \multirow{2}{*}{10} & $L\mapsto 2e+4f$ &    \multirow{2}{*}{4} \\
    & & && & & & $L\mapsto 2f-2h$ & \\
    
    \hline
    30 & 120  &1711 &  $\left(\begin{array}{cc} 20 & 10 \\ 10 & 20 \end{array}\right)$ & 5 & 5 & 5 & $L\mapsto 3e+f-2h$ & 8 \\
    \hline
     \multirow{3}{*}{40} &  \multirow{3}{*}{160}  &\multirow{3}{*}{3081} & $\left(\begin{array}{cc} 4 & 0 \\ 0 & 4 \end{array}\right)$ & 1 & 0 & 1 & $L\mapsto 2e+2f+4h$ &   40 \\
    & & & $\left(\begin{array}{cc} 20 & 0 \\ 0 & 20 \end{array}\right)$ & 5 & 0 & 5 & $L\mapsto 2e+6f$ &  8\\
    \hline
    
     \multirow{4}{*}{45} &  \multirow{4}{*}{180} &\multirow{4}{*}{3916} & \multirow{2}{*}{$\left(\begin{array}{cc} 20 & 0 \\ 0 & 40 \end{array}\right)$}
     & \multirow{2}{*}{5}& \multirow{2}{*}{0}&\multirow{2}{*}{10} & $L\mapsto 3f-3h$ &  \multirow{2}{*}{6}\\
     & &  &  & &  &  & $L\mapsto 3e+6f$ &   \\
   
    & &&   \multirow{2}{*}{$\left(\begin{array}{cc} 20 & 0 \\ 0 & 360 \end{array}\right)$} &  \multirow{2}{*}{5 }&  \multirow{2}{*}{0} &  \multirow{2}{*}{90} & $L\mapsto 4e+6f+h$ &    \multirow{2}{*}{2} \\
    & & & & & & & $L\mapsto 3e+4f+4h$ & \\
    \hline

    \multirow{8}{*}{90} &   \multirow{8}{*}{360} &   \multirow{8}{*}{15931}& $\left(\begin{array}{cc} 20 & 0 \\ 0 & 20 \end{array}\right)$ & 5 & 0 & 5 & $L\mapsto 3e+9f$  & 12 \\
     & & & $\left(\begin{array}{cc} 4 & 0 \\ 0 & 4 \end{array}\right)$ & 1 & 0 & 1 & $L\mapsto 3e+3f+6h$ &  60 \\
    & & & $\left(\begin{array}{cc} 20 & 0 \\ 0 & 180 \end{array}\right)$ & 5 & 0 & 45 & $L\mapsto 3e+7f-2h$ &  4 \\
    & & & $\left(\begin{array}{cc} 8 & 4 \\ 4 & 20 \end{array}\right)$ & 2 & 2 & 5 & $L\mapsto 3e+3f-4h$ &  20 \\
    \hline
\end{longtable}

\enlargethispage{2cm}
\section{Quadrics for $X\subset\mathbb{P}^7$}\label{app-A}
\normalsize The surface $X$ of Section \ref{L^2=12} is described as the intersection of 10 quadrics in $\mathbb P^7$. We detail now the equations of the quadrics $F_i(x_1,\ldots,x_8)$, with $i\in\{1,\ldots,10\}$, obtained by computations by MAGMA using a program by C. Bonnaf\'e. In what follows, $a$ is a primitive root of unity of order 20. 

{\tiny
\begin{equation*}
\begin{split}F_1:=\textstyle
\frac1{15}&(-736a^7 + 528a^6 - 352a^5 - 528a^4 - 736a^3 - 304)x_1x_7 + 
        \textstyle \frac1{15}(736a^7 + 272a^6 + 192a^5 - 272a^4 + 736a^3 - 576)x_1x_8+\\
        \textstyle\frac1{15}&(64a^7 + 208a^6 + 368a^5 - 208a^4 + 64a^3 - 64)x_2x_7 + 
        \textstyle\frac1{15}(-16a^7 - 192a^6 - 112a^5 + 192a^4 - 16a^3 + 416)x_2x_8 +\\
        \textstyle\frac1{15}&(-176a^7 - 32a^6 - 112a^5 + 32a^4 - 176a^3 + 176)x_3x_5 +
        \textstyle\frac15(-32a^7 + 16a^6 - 64a^5 - 16a^4 - 32a^3 - 128)x_3x_6 + \\
        \textstyle\frac15&(48a^7 + 96a^6 + 96a^5 - 96a^4 + 48a^3 - 128)x_4x_5 + 
        \textstyle\frac15(48a^7 - 64a^6 + 16a^5 + 64a^4 + 48a^3 + 352)x_4x_6,
 \end{split}\end{equation*}
    
\begin{equation*}
\begin{split}F_2:=&(-224a^7 - 144a^6 - 160a^5 + 144a^4 - 224a^3 +              256)x_1x_7 + (32a^7 + 
        240a^6 + 16a^5 - 240a^4 + 32a^3 - 448)x_1x_8 +\\
        &(16a^7 + 96a^6 + 
        16a^5 - 96a^4 + 16a^3 - 128)x_2x_7 + (64a^7 - 80a^6 + 32a^5 + 
        80a^4 + 64a^3 + 112)x_2x_8 +\\& (-16a^7 - 64a^6 - 16a^5 + 64a^4 - 
        16a^3 + 128)x_3x_5 + (-96a^7 - 16a^6 - 48a^5 + 16a^4 - 
        96a^3)x_3x_6 +\\& (-32a^7 + 112a^6 - 112a^4 - 32a^3 - 192)x_4x_5 + 
        (192a^7 - 48a^6 + 96a^5 + 48a^4 + 192a^3 + 144)x_4x_6, \end{split}\end{equation*}

\begin{equation*}
    \begin{split}F_3:=&
    (-96a^7 - 16a^6 - 48a^5 + 16a^4 - 96a^3 + 32)x_1x_5 + \textstyle\frac13(-32a^7 + 
        144a^6 - 80a^5 - 144a^4 - 32a^3 - 176)x_1x_6 + \\&\textstyle\frac13(32a^7 + 112a^6
        - 112a^4 + 32a^3 - 240)x_2x_5 + \textstyle\frac13(80a^7 + 32a^6 + 64a^5 - 32a^4
        + 80a^3 + 16)x_2x_6 +\\& (-16a^7 - 16a^3)x_3x_7 + \textstyle\frac13(32a^7 + 48a^6 +
        80a^5 - 48a^4 + 32a^3 - 16)x_3x_8 + (-16a^7 - 16a^5 - 16a^3 + 
        16)x_4x_7 +\\& (16a^7 + 32a^6 - 32a^4 + 16a^3 + 16)x_4x_8,\end{split}\end{equation*}
        
\begin{equation*}
    \begin{split}F_4:=&
    \textstyle\dfrac13(-112a^7 - 256a^6 - 32a^5 + 256a^4 - 112a^3 + 400)x_1x_5 + 
        \textstyle\frac13(-144a^7 + 64a^6 - 32a^5 - 64a^4 - 144a^3 - 32)x_1x_6 +\\& 
        \textstyle\frac13(-96a^7 + 112a^6 - 80a^5 - 112a^4 - 96a^3 - 176)x_2x_5 + 
        \textstyle\frac13(48a^7 + 32a^6 + 80a^5 - 32a^4 + 48a^3 - 112)x_2x_6 + \\&
        \textstyle\frac13(-16a^6 - 16a^5 + 16a^4 + 80)x_3x_7 + \textstyle\frac13(16a^7 + 32a^6 - 
        32a^4 + 16a^3 - 96)x_3x_8 + (-16a^7 - 32a^6 + 32a^4 - 16a^3 + 
        80)x_4x_7 +\\& (-16a^7 + 32a^6 + 16a^5 - 32a^4 - 16a^3 - 48)x_4x_8,\end{split}\end{equation*}
        
\begin{equation*}
    \begin{split}F_5:=&
        (32a^7 - 160a^6 + 160a^4 + 32a^3 + 192)x_1x_4 + 32x_2x_3 + (-32a^7 + 
        32a^6 - 64a^5 - 32a^4 - 32a^3 - 32)x_2x_4 + \\&(-64a^6 - 16a^5 + 
        64a^4 + 128)x_5^2 + (-96a^7 + 32a^6 - 64a^5 - 32a^4 - 96a^3 - 
        64)x_5x_6 + (64a^6 + 16a^5 - 64a^4 - 32)x_6^2 +\\& (32a^7 - 32a^6 + 
        16a^5 + 32a^4 + 32a^3 + 32)x_7^2 + (-96a^7 - 32a^6 - 64a^5 + 
        32a^4 - 96a^3)x_7x_8 +\\& (32a^7 + 32a^6 + 48a^5 - 32a^4 + 32a^3 - 
        64)x_8^2,\end{split}\end{equation*}
        
\begin{equation*}
    \begin{split}F_6:=&
    32x_1x_3 + (32a^7 - 32a^6 + 64a^5 + 32a^4 + 32a^3 + 32)x_1x_4 + 
        (-32a^7 - 32a^6 + 32a^4 - 32a^3)x_2x_4 +\\&  (16a^7 - 48a^6 + 48a^4 +
        16a^3 + 64)x_5^2 +(-64a^7 - 32a^5 - 64a^3)x_5x_6 + (16a^7 + 16a^6
        - 32a^5 - 16a^4 + 16a^3 - 32)x_6^2 +\\& (16a^7 + 16a^6 - 16a^4 + 
        16a^3)x_7^2 + (-32a^5 + 64)x_7x_8 + (16a^7 + 16a^6 - 32a^5 - 16a^4
        + 16a^3 - 32)x_8^2,
        \end{split}\end{equation*}

\begin{equation*}
    \begin{split}F_7:=&   \textstyle\frac15(-304a^7 + 352a^6 - 128a^5 - 352a^4 - 304a^3 - 496)x_1^2 + (64a^7 
        + 32a^5 + 64a^3 + 32)x_1x_2 + (-16a^7 - 16a^3 + 16)x_2^2 + \\&
        \textstyle\frac15(-48a^7 - 96a^6 - 96a^5 + 96a^4 - 48a^3 + 48)x_3^2 + 
        \textstyle\frac15(-656a^7 + 128a^6 + 128a^5 - 128a^4 - 656a^3 + 16)x_4^2 + \\&
        \textstyle\frac15(608a^7 - 64a^6 + 256a^5 + 64a^4 + 608a^3 + 672)x_5x_7 + 
        \textstyle\frac15(-384a^7 - 288a^6 - 448a^5 + 288a^4 - 384a^3 + 384)x_5x_8 + \\&
        \textstyle\frac15(-192a^7 - 544a^6 - 64a^5 + 544a^4 - 192a^3 + 512)x_6x_7 + 
        \textstyle\frac15(-224a^7 + 192a^6 + 192a^5 - 192a^4 - 224a^3 - 736)x_6x_8,\end{split}\end{equation*}
        
\begin{equation*}\begin{split}
    F_8:=&\textstyle\frac15(-256a^7 + 368a^6 - 32a^5 - 368a^4 - 256a^3 - 704)x_1^2 + 
        \textstyle\frac15(288a^7 - 224a^6 + 96a^5 + 224a^4 + 288a^3 + 192)x_1x_2 + \\&
        \textstyle\frac15(-32a^7 + 16a^6 - 64a^5 - 16a^4 - 32a^3 + 32)x_2^2 + 
        \textstyle\frac15(-128a^7 - 16a^6 + 64a^5 + 16a^4 - 128a^3 - 32)x_3^2 + \\&
        \textstyle\frac15(32a^7 - 96a^6 - 256a^5 + 96a^4 + 32a^3 - 32)x_3x_4 + 
        \textstyle\frac15(-224a^7 + 272a^6 - 448a^5 - 272a^4 - 224a^3 - 416)x_4^2 +\\& 
        \textstyle\frac15(352a^7 - 576a^6 + 224a^5 + 576a^4 + 352a^3 + 768)x_5x_7 + 
        \textstyle\frac15(-608a^7 + 64a^6 - 256a^5 - 64a^4 - 608a^3 + 128)x_5x_8 + \\&
        \textstyle\frac15(-448a^7 - 96a^6 - 96a^5 + 96a^4 - 448a^3 + 448)x_6x_7 + 
        \textstyle\frac15(192a^7 + 544a^6 + 64a^5 - 544a^4 + 192a^3 - 512)x_6x_8,\end{split}\end{equation*} 
        
\begin{equation*}\begin{split}F_9:=
&    \textstyle\frac19(-160a^7 - 352a^5 - 160a^3 - 16)x_1^2 + (32a^7 + 32a^6 - 32a^4 + 
        32a^3 - 32)x_1x_2 + 16x_2^2 + \textstyle\frac19(-32a^6 - 32a^5 + 32a^4 + 16)x_3^2
        + 80x_4^2 +\\& \textstyle\frac19(128a^7 - 64a^6 - 128a^5 + 64a^4 + 128a^3 + 
        160)x_5x_7 + \textstyle\frac19(-32a^7 + 96a^6 - 32a^5 - 96a^4 - 32a^3 - 
        224)x_5x_8 + \\&\textstyle\frac19(-32a^7 + 32a^6 - 96a^5 - 32a^4 - 32a^3 + 
        96)x_6x_7 + \textstyle\frac19(128a^7 - 64a^5 + 128a^3 - 160)x_6x_8,\end{split}\end{equation*} 
        
\begin{equation*}\begin{split}F_{10}:=
&    (-96a^7 + 96a^6 - 64a^5 - 96a^4 - 96a^3 - 144)x_1^2 + (96a^7 - 32a^6 
        + 64a^5 + 32a^4 + 96a^3 - 32)x_1x_2 +\\& (-32a^6 - 32a^5 + 32a^4 + 
        16)x_2^2 + 80x_3^2 + (-32a^6 + 32a^5 + 32a^4 + 16)x_4^2 - 32x_5x_7 + 
        (32a^7 - 32a^6 - 32a^5 + 32a^4 + 32a^3 + 32)x_5x_8 +\\& (32a^7 + 
        32a^6 + 32a^5 - 32a^4 + 32a^3 - 32)x_6x_7 + (-64a^6 - 64a^5 + 
        64a^4 + 32)x_6x_8\end{split}\end{equation*}}

\bibliographystyle{plain}
\bibliography{references}

\end{document}